\documentclass[11pt]{amsart}
\usepackage{amssymb, amsmath}
\usepackage{amscd}
\usepackage{amsthm}
\usepackage{graphicx}
\newtheorem{theorem}{Theorem}
\newtheorem{lemma}[theorem]{Lemma}

\theoremstyle{definition}

\newtheorem*{acknowledgement}{Acknowledgement}
\title{Smooth (non)rigidity of piecewise rank one locally symmetric manifolds}
\author{T. T$\hat{\mathrm{a}}$m Nguy$\tilde{\hat{\mathrm{e}}}$n Phan}
\DeclareMathOperator{\SL}{SL}

\DeclareMathOperator{\Sol}{Sol}

\DeclareMathOperator{\Ker}{Ker}

\DeclareMathOperator{\cd}{cd}

\DeclareMathOperator{\Isom}{Isom}
\DeclareMathOperator{\Out}{Out}

\DeclareMathOperator{\PSL}{PSL}
\DeclareMathOperator{\Aut}{Aut}
\DeclareMathOperator{\Fix}{Fix}

\DeclareMathOperator{\Diff}{Diff}
\DeclareMathOperator{\Conj}{Conj}

\def\R{\mathbb{R}}
\def\Z{\mathbb{Z}}
\def\N{\mathbb{N}}
\def\Q{\mathbb{Q}}

\def\H{\mathbb{H}}

\input xy
\xyoption{all}
\begin{document}
\begin{abstract}
We define \emph{piecewise rank $1$} manifolds, which are aspherical manifolds that generally do not admit a nonpositively curved metric but can be decomposed into pieces that are diffeomorphic to finite volume, irreducible, locally symmetric, nonpositively curved manifolds with $\pi_1$-injective cusps. We prove smooth (self) rigidity for this class of manifolds in the case where the gluing preserves the cusps' homogeneous structure.  We compute the group of self homotopy equivalences of such a manifold and show that it can contain a normal free abelian subgroup and thus, can be infinite. Elements of this abelian subgroup are twists along elements in the center of the fundamental group of a cusp. 
\end{abstract}
\maketitle
\section{Introduction and the main theorems}
Let $M$ be a smooth manifold. We denote by $\Diff(M)$ the group of self-diffeomorphisms of $M$ and by $\Diff_0(M)$ the group of self-diffeomorphisms of $M$ that are homotopic (not necessarily isotopic) to the identity map. Hence $\Diff(M)/\Diff_0(M)$ is the group of self diffeomorphisms of $M$ up to homotopy. Let $\Out(\pi_1(M))$ be the group of outer automorphisms of $\pi_1(M)$. The action of a diffeomorphism on $\pi_1(M)$ induces a natural homomorphism \[\eta \colon \Diff(M)/\Diff_0(M) \longrightarrow \Out(\pi_1(M)).\]
If $M$ is aspherical, then $\Out(\pi_1(M))$ is canonically isomorphic to the group of self homotopy equivalences of $M$ up to homotopy, and $\eta$ is an injection. One can ask if $\eta$ is a surjection, i.e. if a homotopy equivalence of $M$ is homotopic to a diffeomorphism.

It is known for a number of classes of manifolds that $\eta$ is an isomorphism. These include closed surfaces, by the Dehn-Nielsen-Baer Theorem; infra-nil manifolds, by Auslander \cite{Auslander}; and finite  volume, complete, irreducible, locally symmetric, nonpositively curved manifold of dimension greater than $2$, by Mostow-Prasad-Margulis Rigidity. If we relax the smoothness condition and replace diffeomorphisms by homeomorphisms, this also holds for closed, nonpositively curved manifolds of dimension greater than $4$, by Farrell and Jones \cite{FJ2}, and for solvmanifolds, by work of Mostow \cite{Mostow}.

Observe that the above classes of manifolds are locally homogeneous or nonpositively curved. A natural question one can ask is whether the above rigidity phenomena holds for other classes of aspherical manifolds that are not locally homogeneous and do not admit a nonpositively curved metric. 

In this paper, we study rigidity properties of manifolds obtained by gluing, possibly infinitely many, manifolds with boundary that are compactifications of noncompact, complete, finite volume, irreducible, locally symmetric, nonpositively curved manifolds along their boundaries. We call these manifolds \emph{piecewise rank $1$} manifolds.  The precise definition of these manifolds will be given in Section \ref{piecewise Q-rank 1}. As we will see, these manifolds generally do not admit any locally homogeneous or nonpositively curved metric. 

From now on, by a \emph{locally symmetric} manifold we will mean a complete, finite volume, irreducible, locally symmetric, nonpositively curved manifold.

If $M$ be a locally symmetric manifold, then $M$ is the interior of a compact manifold $\overline{M}$ with boundary, which can be taken to be the manifold obtained by deleting the ends of the cusps of $M$. By identifying a collection of such manifolds along their boundaries, e.g. simply doubling $M$, we get a new manifold. For example, we can take $M$ to be a Hilbert modular surface. In this case, $M$ is the quotient of the product of two hyperbolic planes $\H^2\times\H^2$ by a torsion free subgroup of $\SL(2,\Z[\sqrt{d}])$, for some square free integer $d$ (see Section \ref{examples} for a more detailed description of these manifolds). The boundary of $M$ is a compact manifold with $\Sol$ geometry.

In the spirit of the Borel conjecture, and because there are non-aspherical manifolds that are not topologically rigid, e.g. some lens spaces, we are interested in only aspherical manifolds of the above type. The aspherical condition for these manifolds is that each boundary component is aspherical and is $\pi_1$-injective, which corresponds exactly to the noncompact, locally symmetric manifolds with $\R$-rank $1$ or with higher $\R$-rank but $\Q$-rank $1$.  We call these manifolds \emph{rank $1$} manifolds. The above example of Hilbert modular surfaces are rank $1$ manifolds. The compactification of these manifolds into compact manifolds with boundary can be obtained by deleting the ends of their \emph{cusps}.

We can obtain aspherical manifolds from gluing locally symmetric manifolds of higher $\Q$-rank, but in this case the construction is, however, more sophisticated, for we cannot naively glue those together along each pair of boundary components to get an aspherical manifold. This is because the boundary is either not aspherical or is not $\pi_1$-injective. The case of gluing higher $\Q$-rank manifolds are rather different and can be found in \cite{Tampwlocsym}. 

The first main result of this paper is that if the gluing is by diffeomorphisms that respect the structure of the parabolic subgroups corresponding to the cusps, or \emph{p-affine}, which we will define precisely later, then smooth (self) rigidity holds for these manifolds. 

\begin{theorem}[Smooth rigidity]\label{Mostow}
Let $M$ be a piecewise rank $1$ manifold of dimension $\geq 3$. Suppose that the gluing along pairs of boundary components of the pieces in the cusp decomposition of $M$ is p-affine. Then any homotopy equivalence of $M$ is homotopic to a diffeomorphism. 
\end{theorem}

We do not expect the above theorem to hold if ``p-affine" is removed from the hypothesis for the same reason that there is no smooth Borel conjecture. The ``exotic sphere trick"  in \cite{FJ3} gives smooth manifolds for which the homomorphism $\eta$ is not surjective in certain cases. However, we do not address the general question of for which gluing maps smooth rigidity holds for piecewise rank $1$ manifolds in this paper.

Given the above theorem, one can ask for more, namely if there is a solution to the Nielsen Realization problem for piecewise rank $1$ manifolds, that is, whether the group $\Out(\pi_1(M))$ cannot be realized as a subgroup of $\Diff(M)$. In general, this is not true. Morita \cite{Morita} proved that if $M$ is a surface of genus greater than $17$, the group $\Out(\pi_1(M))$ does not lift to $\Diff(M)$. However, for a piecewise rank $1$ manifold $M$ of dimension $\geq 3$, the group $\Out(\pi_1(M))$ can be realized by a group of diffeomorphisms of $M$.  

\begin{theorem}\label{Nielsen}
Let $M$ be a piecewise rank $1$ manifold of dimension $\geq 3$. Suppose that the gluing along pair of boundary components of the pieces in the decomposition of $M$ is p-affine. Then $\Out(\pi_1(M))$ can be realized a group of diffeomorphisms. That is, there is an injective homomorphism $\rho \colon \Out(\pi_1(M)) \longrightarrow \Diff(M)$ such that 
\[ \xymatrix{
& & \Diff(M) \ar[d]^p  \\
 \Out(\pi_1(M)) \ar[urr]^\rho &  & \Diff(M)/\Diff_0(M)\ar[ll]_\eta  } \]
\end{theorem}

Even though Nielsen realization holds for these manifolds, not every such manifold admits a Riemannian metric such that $\Out(\pi_1(M)) \cong \Isom(M)$, in contrast to the case of Mostow-Prasad-Margulis rigidity. Hence smooth rigidity is the best we can ask for in the case of closed piecewise rank $1$ manifolds.

\begin{theorem}\label{weak rigidity}
Let $M$ be a closed piecewise rank $1$ manifolds of dimension $\geq 3$. If $\Out(\pi_1(M))$ is infinite, then $M$ does not admit any metric such that $\Out(\pi_1(M))$ can be realized as a group of isometries.
\end{theorem}
It turns out that there exist piecewise rank $1$ manifolds $M$ with infinite $\Out(\pi_1(M))$, unlike the locally symmetric case, for which the group $\Out(\pi_1(M))$ is finite. We will give examples of such manifolds.

In proving Theorem \ref{Mostow}, we obtain the following rigidity property of the pieces in decomposition of $M$ under a self homotopy equivalence of $M$.

\begin{theorem}\label{piece rigidity}
Let $M$ be a piecewise rank $1$ manifold of dimension $n > 2$. Let $M_i$, $i\in I$, be the locally symmetric pieces in the decomposition of $M$. Let $f \colon M \longrightarrow M$ be a homotopy equivalence. Then the restriction $f_i$ to $M_i$ is homotopic to a map $g \colon M_i \longrightarrow M$ that is a diffeomorphism onto a piece $M_j \subset M$. 
\end{theorem}
That is, any homotopy equivalence of $M$ must preserve the decomposition up to homotopy. This is the first step in proving Theorem \ref{Mostow}. Pushing to a global diffeomorphism is the other step.

We obtain the following description of the group of self homotopy equivalences of $M$, or $\Out(\pi_1(M))$. By an open cusp of $M$ we mean a cusp of a piece $M_i$ in the decompostion of $M$ that is not glued to another cusp. 
\begin{theorem}[Computation of $\Out(\pi_1(M))$]\label{Out(pi_1(M))}
Let $M$ be a piecewise rank $1$ manifold of dimension $\geq 3$. Then the outer automorphism group $\Out(\pi_1(M))$ is an extension of a free abelian group $\mathcal{T}(M)$ by a group $\mathcal{A}(M)$, i.e. the following sequence is exact.
\[1 \longrightarrow \mathcal{T}(M) \longrightarrow \Out(\pi_1(M)) \longrightarrow \mathcal{A}(M) \longrightarrow 1.\] 
The group $\Out(\pi_1(M))$ is infinite if the fundamental group of one of the non-open cusps in the decomposition of $M$ has nontrivial center. In fact, $\mathcal{T}(M)$ is isomorphic to the direct sum of the centers of the non-open cusp groups of $M$. If $M$ has finitely many locally symmetric pieces, then $\mathcal{T}(M)$ is finitely generated and $\mathcal{A}(M)$ is finite.
\end{theorem}

The groups $\mathcal{T}(M)$ and $\mathcal{A}(M)$ in the above theorem will be defined in Section \ref{twists and turns}. The elements of $\mathcal{T}(M)$ are called \emph{twists}, which are similar to Dehn twists in surface topology, around loops in the center of a cusp subgroup. It is these twists that prevent $\Out(\pi_1(M))$ to be realized as a group of isometries. Indeed, if $M$ is a closed piecewise rank $1$ manifold of dimension $\geq 3$ with centerless cusp subgroups, then Theorem \ref{Out(pi_1(M))} implies that $\Out(\pi_1(M))$ is finite. By Theorem \ref{Nielsen}, $\Out(\pi_1(M)$ can be realized as finite group of diffeomorphisms and thus, a group of isometries with respect to some metric on $M$. 

One can find an abundance of examples of locally symmetric manifolds with centerless cusp subgroups. For instance, the Hilbert modular manifolds has centerless cusp subgroups. In general, if one of the simple factors of the isometry group of the locally symmetric manifold $M$ is of $\R$ rank $1$, then the cusp subgroups of $M$ is centerless.

Theorems \ref{Mostow}, \ref{Nielsen}, \ref{piece rigidity} and \ref{Out(pi_1(M))} have been proved for the case of gluing complete, finite volume, locally symmetric manifolds with negative curvature, i.e. those with $\R$-rank $1$, in \cite{Tamrigidity}, where we called them \emph{cusp-decomposable} manifolds and proved smooth rigidity within this class of manifolds. Hence, in proving the above theorems, we will focus only on the case of gluing other rank $1$ manifolds, that is, those with $\R$-rank $\geq 2$ and $\Q$-rank $1$.  

In \cite{Tamrigidity}, we rely heavily on virtual nilpotency of the cusp and that the pieces are negatively curved. These no longer hold for the higher $\R$-rank case. The cross section of a cusp of a $\Q$-rank $1$ manifold is not always an infra-nilmanifold like the $\R$-rank $1$ case, but is finitely covered by a fiber bundle with fiber a compact nil-manifold and base a compact, nonpositively curved, locally symmetric space. 
However, we still can use nonpositively curved geometry, together with properties of higher rank latices, such as (T) and (FA) in some cases, to obtain the above results about the group of homotopy equivalences of such manifolds. 

This paper is organized as follows. We define piecewise rank $1$ manifolds in Section \ref{piecewise Q-rank 1}. We describe the structure of their fundamental groups and prove Theorem \ref{piece rigidity} in Section \ref{pi_1}. We discuss the structure of $\Out(\pi_1(M))$ and prove Theorem \ref{Out(pi_1(M))} in Section \ref{twists and turns}. We prove Theorems \ref{Mostow}, \ref{Nielsen} and \ref{weak rigidity} in Section \ref{smooth rigidity}. Backgrounds on $\Q$-rank $1$ locally symmetric spaces and properties used in the proofs of the main theorems can be found in Section \ref{Q rank 1}.

\begin{acknowledgement}
I would like to thank my advisor Benson Farb for suggesting the problem and for his guidance and constant encouragement. I would like to thank Ilya Gekhtman and Dave Witte Morris for their generous help with the theory of arithmetic lattices. I would also like to thank Shmuel Weinberger for helpful conversations. I would like to thank Grigori Avramidi and Benson Farb for useful comments on earlier versions of this paper. 

\end{acknowledgement}

\section{$\Q$-rank $1$ locally symmetric spaces}\label{Q rank 1} 
In this section, we will recall or state the properties of $\Q$-rank $1$ manifolds that we will need in order to prove the main theorems.
 
\subsection{General structure}\label{general structure}
Let $G$ be a semisimple, linear, connected algebraic group $G$ defined over $\Q$ that have $\Q$-rank equal to $1$, i.e. a maximal $\Q$ split torus $A$ of $G$ is isomorphic to $\R$. Let $K$ be a maximal compact subgroup of $G$, and let $\Gamma$ be an arithmetic lattice of $G$, i.e. $\Gamma$ is commensurable to the $\Z$ points $G_\Z$ of $G$. 

A manifold $M$ is an arithmetic manifold of $\Q$-rank $1$ if $M = \Gamma \backslash G/K$, where $G$, $K$, and $\Gamma$ are as above and $\Gamma$ is torsion free. The universal cover $X$ of $M$ is $G/K$, which is a symmetric space of noncompact type. If $x_0$ be a basepoint that corresponds to $K$, then a $\Q$ split torus $A$ represents a geodesic in $X$ going through $x_0$.

$\Q$-rank $1$ manifolds have a thick-thin decomposition like noncompact, complete, finite volume, hyperbolic manifolds. A $\Q$-rank $1$ manifold $M$ has finitely many ends or \emph{cusps}. Each cusp is diffeomorphic to $[0, \infty) \times S$, for some compact $(n-1)$--dim manifold $S$. The fundamental group of each cusp, or each \emph{cusp group}, corresponds to a maximal subgroup of $\pi_1(M)$ of isometries fixing a point on the boundary at infinity of $\widetilde{M}$. Recall that the boundary at infinity $\partial_\infty \widetilde{M}$ of $\widetilde{M}$ is the set of equivalent geodesic rays in $\widetilde{M}$. The parametrization $[0,\infty) \times S$ of a cusp can be taken so that each cross section $a\times S$ is the quotient of a horosphere in $\widetilde{M}$ by the corresponding cusp subgroup. 

The cusps of $M$ correspond to conjugacy classes in $\Gamma$ of $\Q$ proper parabolic subgroups of $G$. Let $P$ be a proper parabolic $\Q$ subgroup of $G$. Since $G$ is $\Q$-rank $1$, the parabolic $P$ is minimal among all $\Q$ parabolic subgroups \cite[Proposition 8.35]{Witte}. There is a $\Q$ split torus $A$ of $G$, and an element $a \in A$ such that 
\[ P = \{g \in G \; | \; \lim_{n \rightarrow \infty} ||a^{-n}ga^n|| < \infty\}.\] 
Geometrically, this means that $P$ is the stabilizer of the point at the boundary at infinity defined by the ray $a^t$ for $t \rightarrow \infty$, which we denote by $a_+$. This is because the distance between $ga^t$ and $a^t$ stays bounded as $t \rightarrow \infty$, that is, $g$ fixes the point $(a_+)_\infty$ on the boundary at infinity corresponding to the ray $a_+$.

Let $N_P$ be the unipotent radical of $P$. Then 
\[N_P = \{g \in G \;|\; \lim_{n \rightarrow \infty} a^{-n}ga^n = e\}.\]
Each element $g$ of $N_P$ is an isometry of $X$ that does not just fix $(a_+)_\infty$ but takes $a_+$ to a geodesic ray $ga_+$ whose distance from $a_+$ is $0$. More precisely, the distance $d(a^t,ga^t) \rightarrow 0$ as $t \rightarrow \infty$. 

The Langlands decomposition of a rational parabolic $P$ is $P = M_PAN_P = N_PM_PA$, where $N_P$ is the unipotent radical of $P$, the group $A$ is a $\Q$ split torus, and $M_P$ is a reductive subgroup of $P$. The groups $M_P$, $A$ and $N_P$ are defined over $\Q$, and $M_PA = C_G(A)$, the center of $A$ in $G$. The latter fact together with $N_P$ being normal in $P$ implies that $A$ normalizes $M_PN_P$. Geometrically, this means that $M_PN_P$ preserves and acts transitively of each horosphere of $X$ centered at $(a+)_\infty$.

The Langlands decomposition of $P$ gives $X$ a \emph{horospherical decomposition} (\cite{Ji})

\[ X = N_P \times A \times X_P,\]
where $X_P$ is the symmetric space corresponding to $M$, i.e. $X_P = M_P/(M_P \cap K)$. Since $M_P$ is reductive, the symmetric space $X_P$ is nonpositively curved.   

Let $\Gamma_P = \Gamma \cap P$ and $\Gamma_{N_P} = \Gamma \cap N_P$. Then $\Gamma_{N_P}$ is an arithmetic subgroup of $N_P$ and $N_P/\Gamma_{N_P}$ is compact. The group $\Gamma_P$ is a cusp group of $\pi_1(M) \cong \Gamma$. The image of $\Gamma_P$ under the projection $p \colon P \longrightarrow P/N_P$ is contained in $M_P$ is an arithmetic lattice of $M_P$ (\cite{BJloc}), which we denoted by $\Gamma_{M_P}$. Then $\Gamma$ has the structure given by the short exact sequence 

\[ 1 \longrightarrow \Gamma_{N_P} \longrightarrow \Gamma_P \longrightarrow \Gamma_{M_P} \longrightarrow 1, \]
which needs not split. The group $\Gamma_{N_P}$ is the fundamental group of a compact nilmanifold and $\Gamma_{M_P}$ is virtually the fundamental group of a compact locally symmetric space with nonpositive curvature.

A $\Q$-rank $1$ arithmetic manifold $M$ has a compactification $\overline{M}$ that is a manifold with boundary whose interior is diffeomorphic to $M$. One can take $\overline{M}$ to be the Borel-Serre compactification of $M$ using the horosphere decomposition (\cite{Ji}). Or one can equivalently take $\overline{M}$ the manifold obtained by deleting the ends of the cusps of $M$. That is, each end of $M$ is homeomorphic to $S \times [0, \infty )$, where $S = X_P/\Gamma_P$ and $[0,\infty) \cong a_+$, so if one deletes $S \times [1,\infty)$, one obtains a compact manifold with boundary with interior diffeomorphic to $M$.

The properties of $\Q$-rank $1$ locally symmetric spaces that we will need in order to prove the results of this paper are the following. We will prove these after stating them.

\begin{enumerate}
\item[a)]The cross section $C$ of a cusp of $M$ is finitely covered by a manifold $\widehat{C}$ that has the structure of a fiber bundle with fiber a compact nilmanifold $F$ and base a compact locally symmetric manifold $B$ with nonpositive curvature that can contain a local Euclidean factor. The fundamental group of $\widehat{C}$ has the structure given by the following extension
\[ 1 \longrightarrow \pi_1(F) \longrightarrow \pi_1(\widehat{C}) \longrightarrow \pi_1(B) \longrightarrow 1, \] 
where $\pi_1(F) = \Gamma_{N_P}$, and $\pi_1(B)$ is the fundamental group of a compact, nonpositively curved, locally symmetric space. By the Godement compactness criterion (Theorem $5.30$ in \cite{Witte}), the group $\pi_1(F)$ is nontrivial. 
\item[b)] Any deck transformation on the universal cover $\widetilde{X}$ that preserves two distinct horospheres corresponds to an element outside $N_P$. 
\item[c)] Each cusp subgroup of $M$ contains a solvable non-virtually-abelian subgroup if $G$ has $\R$-rank greater than $2$.
\item[d)] Any isometry of $M$ extends to a diffeomorphism of $\overline{M}$.
\item[e)] The center of $M_PN_P$ is contained in $N_P$. 
\item[f)] The unipotent radical $N_P$ is a $2$-step nilpotent group. 

\end{enumerate}

\begin{proof}[Proof of Property (a)]
This follows from the Langlands decomposition of $\Q$ parabolic subgroups of $G$. To see this, let 
$\Lambda$ be a finite index, torsion free subgroup of $\Gamma_{M_P}$. Then $\Gamma_{N_P}\rtimes \Lambda$ corresponds to a finite-sheeted cover $\widehat{C}$ of $C$ that has the structure of a fiber bundle with fiber $F = N_P/\Gamma_{N_P}$ and base $B = X_P/\Lambda$. The group $\Gamma_{N_P}$ is a cocompact lattice in $N_P$ (see \cite{Witte}). The manifold $X_P/\Lambda$ is a compact, locally symmetric, nonpositively curved manifold since $\Lambda$ is torsion free. 
\end{proof}
 
\begin{proof}[Proof of Property (b)]
This follows from the fact that two distinct parabolics cannot share any nontrivial unipotent element (\cite{Witte}). If $P'$ is another proper parabolic subgroup of $G$, then $P'$ is rationally conjugate to $P$ and the intersection $P \cap P'$ cannot contain any nontrivial unipotent elements of $P$ or $P'$ by \cite[Theorem 8.27]{Witte}. 
\end{proof}

\begin{proof}[Proof of Property (c)]
Since $G$ has $\R$-rank $\geq 2$, the group $M_P$ in the Langlands decomposition of a proper parabolic $P$ is nontrivial. Let $g$ be an element of $M_\Z$ that is contained in $\Gamma$. We claim that no finite power of $g$ centralizes $N_\Z$, from which it follows that $g\ltimes N_\Z$ is a solvable subgroup that is not virtually abelian. Suppose that (a power of) $g$ centralizes $N_\Z$. Then $g$ centralizes $N$ since $N_\Z$ is Zariski dense in $N$. Let $T$ be a maximal $\Q$ torus of $G$ that contains $t$, and let $\alpha$ be a root of $T$ that occurs in the Lie algebra of $N$. Then $\alpha$ is trivial on $t$, so $-\alpha$ is trivial on $t$. Hence, $t$ centralizes the opposite unipotent $N^-$. Now $N$ and $N^-$ generates $G$. So $g$ is an infinite order element in the center of $G$, which is a contradiction. 
\end{proof}

Property (d) follows from super rigidity (\cite{Zimmer}) and the fact that the $G_\Q$ action on $X$ extends analytically to the Borel-Serre compactification $\overline{X}$ (\cite{Ji}).

\begin{proof}[Proof of Property (e)]
Firstly, we show that any element $g$ in the identity component $Z^0$ of the center $Z$ of $M_PN_P$ that is not in $N_P$ must commute with the Borel subgroup of $G$. This will be a contradiction since the centralizers of Borel subgroups are trivial. Then we will show that $Z$ must be connected.

Let $T$ be a maximal $\R$-split torus and let $\mathfrak{a}$ be the Lie algebra of $T$.  Then the Lie algebra of the unipotent radical $N_P$ of $P$ is 
\[ \mathfrak{n}_I = \sum_{\alpha \in \Phi^+ \setminus \Phi_I} g_\alpha,  \]
where $\Phi^+$ is the set of positive roots and $\Phi_I \subset \Phi^+$ is a set of simple roots. Let
\[\mathfrak{a}_I = \cap_{\alpha\in I}\ker \alpha. \]
Let $\mathfrak{a}^I$ be the orthogonal complement of $\mathfrak{a}_I$. Then the Lie algebra of $M_P$ is
\[ \mathfrak{m}_I = \mathfrak{a}^I \oplus \sum_{\alpha \in \Phi_I} g_\alpha \oplus l,\]  
for $l$ a subspace of the centralizer of $\mathfrak{a}$. (See \cite{BJloc} for details of the above). Now, since $T$ normalizes $M_PN_P$, it normalizes $Z$. Hence, the Lie algebra $\mathfrak{z}$ of $Z^0$ is a direct sum of subspaces of root spaces $g_\alpha$. If $\mathfrak{z}$ contains elements $x$ in a root space $g_\beta$ not in $\Phi_I$. Observe that $\beta$ cannot be $0$ (for $g_\beta$ commutes with $g_\alpha$ for all $\alpha \in \Phi_I$, which means that $g_\beta \subset \mathfrak{a}_I$, which intersects with $\mathfrak{m}_I$ trivially.) 

If $\beta \in \Phi_I$, then $x$ commutes  with $\mathfrak{n}\oplus\mathfrak{m}$ and $\mathfrak{a}_I$ since $\mathfrak{a}_I \subset \ker g_\beta$. Therefore, $x$ commutes with $\Phi^+  \oplus \mathfrak{a}^I$. Thus,  $x$ commutes with a Borel subgroup contained in $P$. Since the centralizer of any Borel subgroup is trivial, this is a contradiction. So $\mathfrak{z}$ is contained in $\mathfrak{n}_I$. Hence, $Z^0$ is contained in the unipotent radical $N_P$ of $P$.

If $g_\beta \subset l$, then $x$ commutes with $\mathfrak{a}$ (since $l$ is a subspace of the centralizer of $\mathfrak{a}_I$) and $\Phi^+$. Hence, $x$ commutes with the Borel subgroup and again, we get a contradiction. 

Now, we prove that $Z$ is connected. Suppose it is not. Then there is an element $g$ in $Z$ such that $g = mn$, for $n \in N_P$ and $m \in M_P -\{1\}$. Let $X_P$ be the symmetric space $(K\cap M_P)\backslash M_P$. Then $m$ acts on $X_P$, and the action of $m$ commutes with $M_P(\Z)$, which is a lattice of $M_P$. Thus, $m$ commutes with the isometry group of $X_P$. 

Since $X_P$ is isometrically a product of symmetric space of noncompact type $Y$ and a Euclidean space $E$, it follows that $m$ must split into a product of two isometries, each of which preserves either $Y$ or $E$. Hence, $m$ acts by a translation on $E$ and by the identity on $Y$. Since each translation can be connected to the identity via a one parameter subgroup (namely translation with the same direction but different displacement), this means that if the above translation is nontrivial, then there is vector in $\mathfrak{z}$ that is not in $\mathfrak{n}_I$, which is a contradiction to above. Hence, $m$ acts on $X_P$ trivially, which implies that $\exp(m) \in K\cap M_P$. This means that $m$ is an isometry of $X_P$ the preserves the fiber $N_P$. Thus, the restriction of $m$ to each fiber has to be an isometry.

By \cite[Theorem 9]{Tamrigidity} and Property (f), the element $m$ has to be in the center of $N_P$, which is a contradiction to the assumption that $m \in M_P$. Therefore, $Z^0$ contains elements outside $N_P$, which is a contradiction.       
\end{proof}

\begin{proof}[Proof of Property (f)]
This follows from the fact that the commutator subgroup $[N_P,N_P]$ is central in $N_P$ (see \cite{Prasad}).  
\end{proof}

\subsection{Examples of manifolds of $\Q$-rank $1$ and $\R$-rank $\geq 2$} \label{examples} We give a few examples this type of manifolds.

\begin{itemize}
\item[1)] \textbf{Hilbert modular manifolds.}

Let $X$ be a product of $n$ copies of the hyperbolic plane, i.e. $X = \H^2 \times \H^2 \times ... \times \H^2$. Let $k$ be a totally real number field and $[k:\Q] = n$. Let $\mathcal{O}_k$ be the ring of integers of $k$ and let $\sigma_1, \sigma_2, ..., \sigma_n$ be the $n$ embeddings of $k$. The group $\PSL(2,\mathcal{O}_k)$ acts on $X$ discretely with finite volume quotient via the following representation
\[ \rho \colon\PSL(2,\mathcal{O}_k) \longrightarrow \PSL(2, \R)^n\]
defined as 
\[ \rho(M) = (\sigma_1(M), \sigma_2(M), ..., \sigma_n(M)).\]
Let $\Gamma$ be a torsion free, finite index subgroup $\PSL(2,\mathcal{O}_k)$. A manifold $M$ is a \emph{Hilbert modular} manifold if it is $X/\Gamma$ for some $X$ and $\Gamma$ as above. The $\R$-rank of $X$ is $n$ and the $\Q$-rank of $M$ is $1$.
\item[2)] \textbf{An example where the cusp groups have nontrivial center: a noncocompact lattice in $\SL(3,\R)$ of $\Q$-rank $1$ (\cite[Section 6E]{Witte}).}

Let $d$ be a square free positive integer. Let $L = \Q[\sqrt{d}]$ be the real quadratic extension of $\Q$ by $\sqrt{d}$. Let $J$ be the following matrix
\[\left(\begin{matrix} 0 & 0 & 1 \\ 0 & 1 & 0 \\ 1 & 0 & 0 \end{matrix}\right).\]
Let $\Gamma = \{g \in \SL(3, \Z[\sqrt{d}] |\; \overline{g}^T J g = J\}$, where $\overline{g}$ is the matrix whose entries are the Galois conjugate of the corresponding entries of $g$. As in \cite[Section 6E]{Witte}, the group $\Gamma$ is a $\Q$-rank $1$ lattice of $\SL(3,\R)$ and the unipotent radical of a maximal $\Q$-parabolic is conjugate to the unipotent upper triangular matrices. That is, in this case,  
\[P = \left(\begin{matrix} * & * & * \\ 0 & * & * \\ 0 & 0 & * \end{matrix}\right), \quad \text{and}\quad N_P = \left(\begin{matrix} 1 & * & * \\ 0 & 1 & * \\ 0 & 0 & 1 \end{matrix}\right).\]
Hence, $M_P \cong \R$, and $\Gamma_{M_P} \cong \Z$. A cusp subgroup $C$ of $\Gamma$ has the structure of the following (split) exact sequence (since $\Z$ is free) 
\[ 1 \longrightarrow H \longrightarrow C \longrightarrow \Z \longrightarrow 1,\]
where $H$ is a lattice of the $3$-dimensional Heisenberg group (i.e. the unipotent upper triangular $3 \times 3$ matrices). Since the center of $H$ is characteristic and is isomorphic to $\Z$, by taking finite index subgroup $\Gamma'$ of $\Gamma$, we can make a cusp group $C'$ of $\Gamma'$ to have the following structure split exact sequence 
\[ 1 \longrightarrow H \longrightarrow C' \longrightarrow \Z \longrightarrow 1,\]
where $\Z$ acts on the center $Z(H)$ of $H$ trivially. It follows that the center of the cusp group of $\Gamma'$ contains $Z(H) \ne 1$.   
\end{itemize}

\subsection{Deck transformations and subgroups of $\pi_1(M)$}
In this section, we will give a few properties of deck transformations on the universal cover of a $\Q$-rank $1$ arithmetic manifold $M$ that we will need in proving our results.

Let $M$ be an arithmetic manifold of $\Q$-rank $1$. For each cusp subgroup of $M$, the corresponding deck transformations on the universal cover $\widetilde{M}$ preserves a family of \emph{concentric} horospheres, i.e. these horospheres have the same center in the boundary at infinity $\partial_\infty M$. We call this center point a \emph{cusp point at infinity}. As a general rule in this paper, whenever horospheres are mentioned in this paper, they are the ones that are preserved by some cusp subgroup of $M$. 

For any two distinct cusp points at infinity $x$ and $y$ of $\widetilde{M}$, there is a geodesic connecting $x$ and $y$. This is because there is a point $p\in\widetilde{M}$ and a flat containing the geodesic rays $\gamma_{px}$ connecting $p$ with $x$ and $\gamma_{py}$ connecting $p$ with $y$ (\cite[Proposition 2.21.14]{Eberlein}). If the angle $\angle xpy$ is less that $\pi$, then any horospheres $H_x$ centered at $x$ and $H_y$ centered at $y$ have nonempty intersection. But this contradicts the fact that the projections of $H_x$ and $H_y$ in the quotient $M$ (which are cross sections of cusps of $M$) can be made disjoint by going far enough down the cusp.

It follows that a parabolic deck transformation fixes only one cusp point at infinity of $\widetilde{M}$. This is because if a parabolic deck transformation $g$ fixes two distinct cusp points $x$ and $y$, then let $\gamma$ be a geodesic connecting $x$ to $y$. Then $g$ preserves set of all (parallel) geodesics connecting $x$ and $y$, which form a totally geodesic set $W$ that is a metric product $\gamma \times U$, where $\gamma$ is a geodesic connecting $x$ and $y$ and $U$ is a convex set. Now $g$ acts on $W$ preserving each slice of $U$ (i.e. $t\times U$) since $g$ preserves horospheres centered at $\gamma(\infty) = x$. The restriction of $g$ to $U$ is an isometry of $U$ that must be semisimple because the quotient of the horosphere centered at $x$ by the corresponding cusp group is compact. This means that the displacement distance of $g$ is realized (and is the translation distance in $W$). But this contradicts the fact that $g$ is parabolic. 

\begin{lemma}\label{cusp embedding}
Let $M^n$ and $N^n$ be  $\Q$-rank $1$, arithmetic manifolds of dimension $n \geq 3$, and assume that $M$ is irreducible. Let $C$ be a cusp of $M$. Let $\phi \colon \pi_1(C) \longrightarrow \pi_1(N)$ be an injective homomorphism. Then $\phi(\pi_1(C))$ has a finite index subgroup that is a finite index subgroup of a cusp subgroup of $N$.  
\end{lemma}

\begin{proof}
Let $P$ be the parabolic corresponding to $C$. Let $\widehat{C}$ be a finite sheeted cover of $C$ as above. We have
\[ 1 \longrightarrow \pi_1(F) \longrightarrow \pi_1(\widehat{C}) \longrightarrow \pi_1(B) \longrightarrow 1,\]
where $F$ and $B$ are as above. We claim that $\pi_1(C)$ has a nontrivial abelian normal subgroup $A$. One can pick $A$ to be the last nontrivial term in the lower central series of $\pi_1(F)$. Since $A$ is characteristic in $\pi_1(F)$ and $\pi_1(F)$ is normal in $\pi_1(\widehat{C})$, it follows that $A$ is normal in $\pi_1(\widehat{C})$. 

Let $H$ be such that $\phi(H)$ is the subgroup of $\phi(A)$ consisting of semisimple isometries. Then by \cite[Lemma 7.7]{BGS}, the group $H$ is normal in $\widehat{C}$, and the Minset $W$ of $\phi(H)$ in $N$ is nonempty and is $\phi(\widehat{C})$-invariant. 

If $H$ is not the trivial group, then $W$ is a totally geodesic submanifold of $N$ of strictly smaller dimension. The action of $\phi(C)$ on $W$ is free and properly discontinuous. Since the cohomological dimension of $C$  is  $(n - 1)$, it follows that the quotient of $W$ by $\phi(C)$ is a compact manifold, which has nonpositive curvature since $W$ is totally geodesic in $N$. Thus, any solvable subgroup of $C$ is virtually abelian (\cite[Theorem 7.16]{Bridson}). But by Fact (c) in the previous section, the group $C$ has a solvable subgroup that is not virtually abelian since $M$ is irreducible. Hence, $H = \{1\}$. Therefore, $\phi(A)$ contains only parabolic isometries. 

Since each parabolic fixes only one cusp point at infinity, and $A$ is abelian, the parabolic isometries in $\phi(A)$ fixes the same cusp point at infinity.  Hence, $\phi(\pi_1(\widehat{C}))$ fixes the same cusp point by normality of $A$ in $C$. So $\phi(\pi_1(\widehat{C}))$ is contained in a cusp subgroup $D$ of $N$. Now, the cohomological dimension $\cd(\pi_1(\widehat{C}))$ is $(n-1)$, which is equal to the dimension of the cross section of a cusp in $N$. It follows that $\phi(\pi_1(\widehat{C}))$ is finite index in $D$.    
\end{proof}

\begin{lemma}\label{projection}
Let $M = K \backslash G/ \Gamma$ be an irreducible, locally symmetric manifold whose universal cover is $X\times Y$ for some symmetric spaces $X$ and $Y$ of noncompact type. Then the projection $p \colon \pi_1(M) \longrightarrow \Isom(X)$ of the group of deck transformation to $\Isom(X)$ is injective. 
\end{lemma}

\begin{proof}
Since $G$ has $\R$-rank $\geq 2$, its lattices have the normal subgroup property. So $p(\Gamma)$ is finite or $\Ker p$ is finite. If $p(\Gamma)$ is finite, then $\Ker p$ is a lattice of $\Q$ rank $1$ since it finitely covers $M$. But this is not a lattice since $(X\times Y)/\Ker p = X \times (Y/\Ker p)$, which has infinite volume. So $\Ker p$ is finite. But since $\Gamma$ is torsion free, $\Ker p = \{1\}$. Hence, $p$ is an injection.
\end{proof}

Let  $G = G_1\times G_2$, where $G_1 = \Isom^0X$ and $G_2 = \Isom^0Y$, where $X$ and $Y$ are as in the above lemma. Since the projection $G \longrightarrow G_1$ is a homomorphism of algebraic groups, by the Jordan-Chevalley theorem (\cite{Procesi}), the projection of a unipotent (semisimle) element is unipotent (semisimple). Therefore, the projection of a strict parabolic (hyperbolic) isometry is a strict parabolic (semisimple) isometry.

\section{Piecewise rank $1$ manifolds: definitions and examples}\label{piecewise Q-rank 1}
Let $Y$ be a complete, finite volume, nonpositively curved, irreducible, locally symmetric manifold. Then $Y = K\backslash G/\Gamma$, where $G$ is a semisimple Lie group $G$ of noncompact type, $K$ is a maximal compact subgroup of $G$, and $\Gamma$ is a torsion free lattice of $G$. The universal cover $\widetilde{Y}$ is $K\backslash G$, which is a nonpositively curved, complete, symmetric space of noncompact type. The dimension of maximal flats in $\widetilde{Y}$ is the dimension of maximal $\R$-split tori of $G$. The dimension of maximal $\Q$-split tori of $G$ is called the $\Q$-rank of $G$. The $\R$-rank ($\Q$-rank) of $Y$ or $\widetilde{Y}$ is defined to be the $\R$-rank ($\Q$-rank) of $G$. 

The manifold $Y$ is negatively curved precisely when the $\R$-rank of $G$ is $1$. If $\widetilde{Y}$ has $\R$-rank $\geq 2$, then by the Margulis arithmeticity theorem (\cite{Zimmer}), the lattice $\Gamma$ is arithmetic. That is, $G$ can be defined over $\Q$ and $\Gamma$ is commensurable to the $\Z$ points $G_\Z$ of $G$. We say that a complete, noncompact, finite volume, locally symmetric manifold $Y$ has \emph{rank $1$} if it has $\R$-rank $1$ or if it has $\R$-rank $\geq 2$ and $\Q$-rank $1$. 

Now let $Y$ be a rank $1$ locally symmetric manifold of noncompact type. Then $Y$ has finitely many ends or \emph{cusps}. Each cusp is diffeomorphic to $[0, \infty) \times S$ for some compact $(n-1)$--dim manifold $S$. The fundamental group of each cusp, or each \emph{cusp subgroup}, corresponds to a maximal subgroup of $\pi_1(Y)$ of parabolic isometries fixing a point on the boundary at infinity of $\widetilde{Y}$. The parametrization $[0,\infty) \times S$ of a cusp can be taken so that each cross section $a\times S$ is the quotient of a horosphere in $\widetilde{Y}$ by the corresponding cusp subgroup. 

If we delete the $(b, \infty) \times S$ part of each cusp $[0, \infty) \times S$ of $Y$, the resulting space is a compact manifold with boundary (that is diffeomorphic to the Borel-Serre compactifition $\overline{Y}$ of $Y$ if $Y$ is arithmetic). We choose $b$ large enough so that the boundary components $b \times S$ of different cusps do not intersect. We require the boundary components of $X$ to lift to a horosphere in $\widetilde{Y}$.

Each boundary component $S$ of $\overline{Y}$ is a locally homogeneous space with respect to a transitive action of a parabolic subgroup $P$ of $G$, which is the stabilize a point in the boundary at infinity $\partial_\infty\widetilde{S}$. A diffeomorphism $f \colon S \longrightarrow S'$ between two boundary components is \emph{p-affine} if there is a Lie group isomorphism $\varphi \colon P \longrightarrow P'$ that induces $f$. That is, the image $\varphi(\pi_1(S)) = \pi_1(S')$.

We say that a manifold $M$ is a \emph{piecewise rank $1$} manifold if it is obtained by taking such manifolds $M_i$ with and glue them along pairs of boundary components via homeomorphisms. Each pair of boundary components that are glued together can belong to the same $M_i$. The set of the spaces $M_i$ with a gluing is called a (the) \textit{decomposition into locally symmetric pieces} of $M$. Each $M_i$ is a \emph{piece} in the cusp decomposition of $M$. Such a manifold $M$ is of \emph{finite type} if there are finitely many pieces in the cusp decomposition of $M$.

A simple example of piecewise rank $1$ manifolds is the double of a $\Q$-rank $1$ locally symmetric space. Another example is the following. If $Y$ is a rank $1$ manifold with at least 2 cusps, let $\overline{Y}$ be its compactification. Suppose that $\overline{Y}$ has two homeomorphic boundary components  $b\times S$ and $b'\times S'$, then we can glue the two boundary components together by a homeomorphism. For example, take $Y$ to be a Hilbert modular manifold. The resulting manifold is a piecewise rank $1$ manifold, and so is any cover of it.

We would also like to describe a few properties of covering spaces of piecewise rank $1$ manifolds and review how their universal covers relate to their Bass-Serre tree. Let $\widehat{M}$ be a covering space of $M$ with covering map $p \colon \widehat{M} \longrightarrow M$. Since $M$ is obtained from gluing $M_i$'s along their boundaries, each component of $p^{-1}(M_i)$ in $\widehat{M}$ is a covering space of $M_i$ for each $i$. Similarly, a component of $p^{-1}(\partial M_i)$ in $\widehat{M}$ is a covering space of $\partial M_i$ for each $i$. Thus, $\widehat{M}$ is the union of connected covering spaces of $Y_i$ glued along their boundaries.

By the above description of covering spaces of $M$, the universal cover $\widetilde{M}$ of $M$ is the union of universal covers of \emph{neutered spaces} $M_i$ (which are complete simply connected locally symmetric negatively curved manifolds with a collection of disjoint horoballs removed) glued to each other along their horosphere boundaries. Each of the neutered spaces is a connected component of $p^{-1} (M_i)$ for some $M_i$ in the cusp decomposition on $M$. The underlying graph in the graph of spaces defined by this decomposition of $\widetilde{M}$ is the Bass-Serre tree of the graph of groups structure of $\pi_1(M)$ defined by the decomposition of $M$ (see the next section for more detailed discussion). Two horospheres that belong to the same neutered space are called \emph{adjacent}. 

The universal cover $\widetilde{M}$ is contractible since it deformation retracts onto the Bass-Serre tree of its fundamental group. This is because for each $i$, the manifold $\widetilde{M_i}$ is contractible and each lift of $\partial M_i$ is contractible since it is a horosphere of a symmetric space of nonpositive curvature.   

\section{The fundamental group of piecewise rank $1$ manifolds}\label{pi_1}

Let $M$ be a piecewise rank $1$ manifold of dimension $n >2$ and let $\{M_i\}_{i \in I}$ be the pieces in the decomposition of $M$ into locally symmetric pieces. The decomposition of $M$ gives $\pi_1(M)$ the structure of the fundamental group of a graph of groups $\mathcal{G}_M$ (see \cite{Serre} and \cite{Scott} for Bass-Serre theory). The vertex groups of $\mathcal{G}_M$ are the fundamental groups of the pieces $M_i$ in the cusp-decomposition of $M$. The edge groups of $\mathcal{G}_M$ are the fundamental groups of the cusps of each piece, i.e. boundary components of each $M_i$. The injective homomorphism from an edge group to a vertex group is the inclusion of the corresponding boundary component into $M_i$.  

\begin{figure}
\begin{center}
\includegraphics[height=100mm]{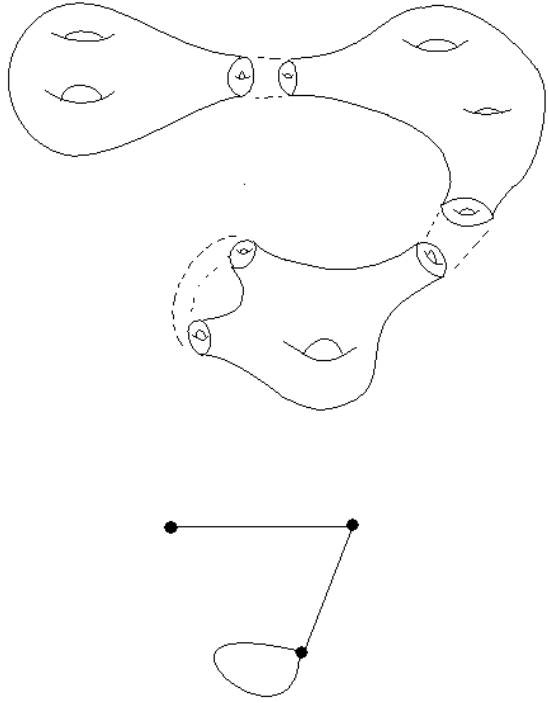}
\caption{A piecewise rank $1$ manifold and the corresponding underlying graph of the graph of groups structure of the fundamental group.}
\end{center}
\end{figure}

\begin{lemma}\label{cusp subgroups}
Let $M$ be a piecewise rank $1$ manifold and let $$\phi \colon \pi_1(M) \longrightarrow \pi_1(M)$$ be an automorphism. If $C$ is an edge group of $\pi_1(M)$, then a finite index subgroup of $\phi(C)$ is contained in a conjugate of an edge group $\pi_1(M)$.
\end{lemma}

Before proving Lemma \ref{cusp subgroups}, we state a result in Bass-Serre theory (\cite[Section 6.5, Proposition 27]{Serre}) which we will use in the proof.

\begin{theorem}[\cite{Serre}]\label{nilpotent action}
Let $G$ be a finitely generated nilpotent group acting on a tree $T$ without inversion. Then either $G$ has a fixed point or there is an axis $L$ that is preserved by $G$, on which $G$ acts by translation by means of a nontrivial homomorphism $G \longrightarrow \Z$.
\end{theorem}

\begin{proof}[Proof of Lemma \ref{cusp subgroups}]
If $C$ is the cusp subgroup of a piece $M_i$ whose universal cover does not contain a copy of real or complex hyperbolic space, then we are done. This is because $\pi_1(M_i)$, which is a vertex group that contains $C$, has property (T) and thus it has property (FA). Hence, the action of $\phi(\pi_1(M_i))$ on the Bass Serre tree of $M$ has a global fixed point. Hence, $\phi(\pi_1(M_i))$ and thus $\phi(C)$ is contained in a conjugate of a vertex group of $\pi_1(M)$. Then by Lemma \ref{cusp embedding}, a finite index subgroup of $\phi(C)$ is a finite index subgroup of a cusp group of $M$.

Suppose that the two pieces that the cusp corresponding to $C$ belong to has reducible universal covers that have a real or complex hyperbolic space factor. Call one of them $X\times Y$ for $X$ real or complex hyperbolic. We consider the action of the group $\phi(C)$ on the Bass-Serre tree $T$ of $\pi_1(M)$. By property (a) in Section \ref{general structure}, we can assume (with abused notation) that $\pi_1(C)$ has the following structure 
\[ 1 \longrightarrow F \longrightarrow C \longrightarrow B \longrightarrow 1.\]
Note the $F$ is maximal unipotent of the vertex group that it is in. By the Godement compactness criterion (Theorem $5.30$ in \cite{Witte}), the group $F$ is nontrivial.

\textbf{Case 1}: The set $\Fix(\phi(F))$ is a bounded-diameter set $S$. Since $F$ is normal in $C$, it follows that $\phi(C)$ preserves the center of mass of $S$, i.e., $\phi(C)$ has a global fixed point. Hence, $C$ is contained in a conjugate of cusp subgroup by Lemma \ref{cusp embedding} and that $M$ is irreducible.

\textbf{Case 2}: The set $\Fix(\phi(F)$ has unbounded diameter, i.e. it contains arbitrarily long paths. Then $\phi(F) \leq C'$ for some cusp group $C'$. Let $N = \phi^{-1}(F')$. Then $F$ normalizes $N$ and thus, preserves the Minset of $N$. Since $F$ contains only parabolic elements as they are unipotent, it does not preserve any long path in $T$ (i.e. paths that contains at least two edges). Hence, the action of $N$ cannot factor through a translation along some (unique) axis. So $\Fix(N)$ is nonempty and can be either a vertex or an edge. 


 
So $N$ preserves at most one horosphere. Since $\phi^{-1}(C')$ normalizes $N$, it follows that $\phi^{-1}(C')$ preserves the $\Fix(N)$, which is at most one edge in $T$. The action on $T$ is without inversion, so $\phi^{-1}(C')$ fixes the Fixset of $N$ pointwise. Hence, $\phi^{-1}(C')$ is contained in a cusp subgroup of $M$. By Lemma \ref{cusp embedding}, $\phi^{-1}(C')$ is virtually contained in the cusp subgroup that contains $F$, which is $C$. Thus, $\phi(C)$ is virtually contained in $C'$. 

\textbf{Case 3}: The action of $\phi(F)$ on $T$ factors through a translation along some axis $l'$ of $T$. Since $F$ is normal in $C$, the action of $\phi(C)$ on $T$ leaves $l'$ invariant. Hence, we get a homomorphism
\[ \alpha \colon C \longrightarrow \Aut(L) \cong \Z\rtimes\Z_2.\] 
Hence, $C$ has a subgroup $H$ of index at most $2$ such that the action of $\phi(H)$ on $l'$ is via orientation preserving automorphisms of $L$ only and that $\phi(H)$ surjects onto a nontrivial group of translations of $L$. Let $H_0 = \Ker\alpha$. Then we have
\[ 1\longrightarrow H_0 \longrightarrow H \longrightarrow \Z \longrightarrow 1.\]

The elements of $\phi(H_0)$ are deck transformations that fixes $l'$ pointwise. Therefore, there exists an edge $e'$ in $T$ whose corresponding group $C'$ contains $\phi(H_0)$. Since $\Z$ is free, the group $H$ is isomorphic to $\Z\ltimes H_0$. Now, $C'$ has the structure
\[ 1 \longrightarrow F' \longrightarrow C' \longrightarrow B' \longrightarrow 1.\]
Let $G_{v'}$ be a vertex group containing $C$. Since $\phi(H_0)$ fixes two horospheres in $G_{v'}$, the projection homomorphism $C' \longrightarrow B'$, when restricted to $\phi(H_0)$, is injective. This means that $\phi(H_0)$ is isomorphic to a subgroup of $B'$. This implies that
\[ \cd(C) = \cd (H) \leq \cd(\phi(H_0)) + 1 \leq \dim(B') + 1 \leq \dim (C') - 1 + 1 =  n - 1 = \cd(C),\]
where $\cd$ denotes the cohomological dimension of a group. The first equality holds because $H$ is a finite index subgroup of $C$. The second inequality holds because $H$ is an extension of $H_0$ by $\Z$ and $\cd(H_0) \leq \dim(B')$, which follows from the fact that $\phi(H_0)$ embeds as a subgroup of $B'$. The third inequality holds because $\cd(F')$ is at least $1$ since $F'$ is a finitely generated, nontrivial, torsion free, nilpotent group. 

So we must have equality in all of the above inequalities, which implies that $\cd(F') = 1$, which implies that $F'\cong\Z$, and $\phi(H_0)$ is finite index in $B'$. 

Now consider $\phi^{-1}$. If $\Fix(\phi^{-1}(F'))$ is nonempty, then the above case $1$ and case $2$ apply and we are done. Suppose that $\phi^{-1}(F')$ is translation along some axis $\l$ in $T$. The same argument applies to give $F \cong \Z$. Now, 

\[ 1 \longrightarrow F_1 \longrightarrow H\longrightarrow H_1 \longrightarrow 1,\]
for some finite index subgroup $H_1$ of $B$ and some finite index subgroup $F_1$ of $F \cong \Z$. So $\phi(F_1)$ is a group of translation along $l'$. There is a finite index subgroup $\widehat{H}$ of $H$ such that $\phi(\widehat{H})$ factors through the same group of translations of $\phi(F_1)$. Then 
\[ 1 \longrightarrow F_1 \longrightarrow \widehat{H} \longrightarrow \widehat{H}_1 \longrightarrow 1,\]
for some finite index $\widehat{H}_1$ of $H_1$. There is a homomorphism from $\widehat{H}$ to $F_1$ that commutes with the inclusion $F_1 \longrightarrow \widehat{H}$, namely by factoring through translation along $l'$ and identifying that with $\phi(F_1)$. It follows that $\widehat{H}\cong \Z\bigoplus \widehat{H}_1$. Since $\widehat{H}_1$ is a finite index subgroup of $B$, it does not have a solvable nonabelian subgroup and thus, neither does $H$. This is a contradiction.

\end{proof}

Now we prove Theorem \ref{piece rigidity}. Since each piece $M_i$ of $M$ is aspherical, we only need to show the statement on the level of fundamental group. That is, it suffices to prove that following lemma.

\begin{lemma}\label{vertex subgroup}
Let $M$ be a piecewise rank $1$ manifold of dimension $n > 2$ and let $\phi \colon \pi_1(M) \longrightarrow \pi_1(M)$ be an isomorphism. If $G_v$ is a vertex group in the graph of groups decomposition of $\pi_1(M)$,  then $\phi(G_v)$ is conjugate to a vertex group of $\pi_1(M)$.
\end{lemma}

\begin{proof}
Fix a basepoint $x_0 \in M$, and pick a lift $\widetilde{x}_0 \in \widetilde{M}$. Without loss of generality, we can assume that $\widetilde{f}(\widetilde{x}_0) = \widetilde{x}_0$. Now, each horosphere in $M$ is the universal of a horoboundary component (which is a compact $(n-1)$-dim manifold) of a piece in the cusp decomposition of $M$. Since $\phi$ maps cusp subgroups to conjugates of cusp subgroups and since $\widetilde{f}$ is a $\phi$--equivariant quasi-isometry, the image of a horosphere under  $\widetilde{f}$ is a bounded distance from a horosphere (or a \emph{quasi-horosphere}), i.e. it is contained in a $d$--neighborhood of a horosphere. Also, since $\phi$ has an inverse, every horosphere in $N$ is the image of a quasi-horosphere in $M$.

We want to prove that any two adjacent horospheres are mapped to two adjacent horospheres. Let $H_a$, $H_b$ be adjacent horospheres. Suppose that $\widetilde{f}(H_a)$ and $\widetilde{f}(H_b)$ are not adjacent quasi-horospheres, i.e. any path starting at a point in $\widetilde{f}(H_a)$ and ending at some point in $\widetilde{f}(H_b)$ crosses a quasi-horosphere $f(H_c)$ for some horosphere $H_c$. 

Let $p \in H_a$ and $q \in H_b$ be such that the distance from $p$ to $H_c$ and the distance from $q$ to $H_c$ is some large enough $K$. Let $\gamma$ be a path connecting $p$ to $q$ such that the distance between $\gamma$ and $H_c$ is at least the minimum of that from $p$ and $q$ to $H_c$. Since $f$ is a quasi-isometry, the distance between $f(\gamma)$ and $f(H_c)$ is positive if $K$ is large enough, i.e. $f(\gamma)$ is a path connecting $f(H_a)$ and $f(H_b)$ that does not cross $f(H_c)$. But this is a contradiction to the assumption that $f(H_a)$ and $f(H_b)$ are separated by $f(H_c)$. So $f$ does not map adjacent horospheres to non-adjacent horospheres. 

Since $f$ has an quasi-isometric inverse, $f$ maps non-adjacent quasi-horospheres to non-adjacent quasi-horospheres. This implies that $f_*$ maps each vertex group to a conjugate of a vertex group, which is what we want to prove.      
\end{proof}
\section{$\Out(\pi_1(M))$: twists and turns}\label{twists and turns}
Let $\mathcal{M}$  be the disjoint union of all the complete locally symmetric spaces corresponding to the pieces $M_i$ in the cusp decomposition of $M$, i.e. they are the spaces before we delete their cusps.
By Theorem~\ref{vertex subgroup}, an element of $\Out(\pi_1(M))$ has to descend to an isomorphism between vertex groups up to conjugation, which, by the Mostow-Prasad Rigidity Theorem, is induced by an isometry with respect to the complete, locally symmetric metric on each of the pieces. The descending map is a homomorphism from $\Out(\pi_1(M))$ to $\Isom(\mathcal{M})$. We call this induced map 
\[\eta \colon \Out(\pi_1(M))\longrightarrow \Isom(\mathcal{M}).\]
Let $\mathcal{A}(M) $ be the image of $\Out(\pi_1(M))$ under $\eta$. We call the elements of $\mathcal{A}(M)$ \emph{turns}. Then $\mathcal{A}(M)$ is the subgroup of $\Isom(\mathcal{M})$ of isometries of $\mathcal{M}$ whose restriction to the boundaries of each pair of cusps that are identified in the cusp decomposition of $M$ are homotopic with respect to the gluing.

If $M$ is finite-volume, complete, locally symmetric of $\Q$ rank $1$, then by the Mostow-Prasad Rigidity Theorem, $\Out(\pi_1(M)) \cong \Isom(M, g_{loc})$, where $g_{loc}$ is the locally symmetric metric on $M$. This implies that $\Out(\pi_1(M))$ is finite since $\Isom(M)$ is finite. One might expect that if $M$ is piecewise rank $1$, given the above theorem, $\Out(\pi_1(M))$ will also be finite. However, this is not true.  

For example, let $M$ be the double of a locally symmetric manifold of $\Q$ rank $1$ with one cusp. Whether or not such manifolds exist is unknown to the author. However, for the purpose of this discussion, this is not a concern because the idea carries over to the general case of piecewise rank $1$ manifolds, it will only involve unnecessary complication in writing concrete automorphisms. Let $M_1$ and $M_2$ be the two pieces in the cusp decomposition of $M$ and let $G_i = \pi_1(M_i)$. Then $\pi_1(M) = G_1*_CG_2$. Pick an element $c_0 \ne 1$ in the center of $C$. Let $\phi \colon \pi_1(M) \longrightarrow \pi_1(M)$ be induced by 
\[\phi(g) = 
\begin{cases}
&  g \; \qquad \qquad \text{if} \; g \in G_1 \\
& c_0gc_0^{-1} \qquad \text{if} \; g \in G_2.
\end{cases}\]
By the universal property of amalgamated products, $\phi$ extends to an automorphism of $G_1*_CG_2$ since $\phi$ is an automorphism when restricted to $G_1$ and $G_2$ and agrees on $C = G_1 \cap G_2$. 

\begin{lemma}\label{twist}
Let $M$ and $\phi$ be as above. Then $\phi$ is an infinite order element of $\Out(\pi_1(M))$.
\end{lemma} 
\begin{proof}

For all $k \in \N$, $\phi^k$ is the identity on $G_1$ and conjugation by $c_0^k$ on $G_2$. Suppose $\phi^k$ is an inner automorphism of $G_1*_CG_2$ for some $k \in \N$. Then there exists $g \in G_1*_CG_2$ such that $\Conj(g)\circ\phi$ is the identity on $G_1*_CG_2$. This implies that for $g_1 \in G_1$, we have $g\phi^k(g_1)g^{-1} = g_1$. So $gg_1g^{-1} = g_1$ since $\phi^k$ is the identity on $G_1$. Thus, $g$ is in the centralizer of $G_1$ in $G_1*_CG_2$. 

We claim that $g$ is in the centralizer of $G_1$. Consider the action of $G_1*_CG_2$ on the Bass Serre tree $T$ of $G_1*_CG_2$. The fixed set of $G_1$ is a vertex $v$. Since $g$ commutes with every element of $G_1$, the fixed set of $G_1$ must be preserved by $g$. Hence, $g$ belongs to the stabilizer of $v$, which is $G_1$. Therefore, $g$ is in the centralizer of $G_1$ and thus $g =1$. This implies that $\phi$ is the identity homomorphism, which is a contradiction since $\phi$ acts nontrivially on $G_2$. Hence, $\phi$ is represents an infinite order element in $\Out(\pi_1(M))$.

\end{proof}

The automorphism $\phi$ above is induced by gluing the identity automorphism on $\pi_1(M_1)$ and an automorphism on $M_2$ that ``twists" the boundary of $\pi_1(M_2)$ around the loop $c_0$ as in the proof of Theorem~\ref{Mostow}. This is analogous to a Dehn twist in surface topology. For each loop $c$ in the center of $C$, we define a \emph{twist} around a loop $c$ to be an automorphism defined as above. By Lemma~\ref{twist}, twists induce infinite order elements of $\Out(\pi_1(M))$. 

Similarly, for each element in the center of a cusp subgroup of a piecewise rank $1$ manifold $M$, we define a corresponding \emph{twist} like above. The induced map of a twist on $\pi_1(M)$ has the same form as $\phi$, that is, up to conjugation, it is the identity on one vertex subgroup and conjugation by an element in the center of the edge subgroup on the other vertex subgroup. 

Let $\mathcal{T}(M)$ be the subgroup of $\Out(\pi_1(M))$ that is generated by twists. It is not hard to see that any two twists commute since either they are twists around loops in disjoint cusps or the loops they twist around commute since they are both in the center of the same cusp subgroup. Hence, $\mathcal{T}(M)$ is a torsion-free abelian group. Indeed, $\mathcal{T}(M)$ is isomorphic to the direct sum of the centers of the cusp subgroups of $M$. We observe that $\mathcal{T}(M)$ is a normal subgroup of $\Out(\pi_1(M))$. We can actually write $\Out(\pi_1(M))$ as a group extension of $\mathcal{T}(M)$ by a group of isometries of a manifold as follows:
\[1 \longrightarrow \mathcal{T}(M) \longrightarrow \Out(\pi_1(M)) \longrightarrow \mathcal{A}(M) \longrightarrow 1.\] 

An isometry of $\mathcal{M}$ permute the set of the pieces $\{\mathcal{M}_i\}$ in the cusp decomposition of $\mathcal{M}$. Thus, for the case where $M$ is of finite type and the cusp decomposition of $M$ has $k$ pieces, there is a homomorphism $\varphi$ from $\Isom(\mathcal{M})$ to the group of permutations of $k$ letters. Let $P$ be the image of $\varphi$. The kernel of $\varphi$ contains precisely those isometries of $\mathcal{M}$ that preserve each of the pieces in $\{\mathcal{M}_i\}$. Hence, $\Isom(\mathcal{M})$ has the structure of an extension of groups as follows: 
\[ 1 \longrightarrow \bigoplus_{i = 1}^k\Isom(\mathcal{M}_i) \longrightarrow \Isom(\mathcal{M}) \longrightarrow P \longrightarrow 1.\]
It follows that $\Isom(\mathcal{M})$ is finite since it has a finite index subgroup that is isomorphic to the direct sum of the isometry groups of the components of $\mathcal{M}$, which are finite. Thus, $\mathcal{A}(M)$ is finite if $M$ is of finite type. Now we prove Theorem~\ref{Out(pi_1(M))}.
\begin{proof}[Proof of Theorem \ref{Out(pi_1(M))}]
It suffices to prove that $\mathcal{T}(M) = \Ker\eta$. By definition, the action of a twist on the fundamental group of each of the pieces in the cusp decomposition of $M$ is the identity map up to conjugation by an element in the center of a cusp subgroup. By Mostow rigidity, the image of a twist under $\eta$ is the identity isometry.  Thus, $\mathcal{T}(M) \leq \Ker\eta$. 

Now if $\phi \in \Ker\eta$, then the isometry that $\phi$ induces on $\mathcal{M}$ is the identity map. Hence, for each $i$, the restriction of $\phi$ on $\pi_1(M_i)$ is the identity map in $\Out(\pi_1(M_i))$. Let $\phi_i$ be the restriction of $\phi$ to $\pi_1(M_i)$. If $M_i$ and $M_j$ are adjacent pieces that are glued together along $S$, then the restriction of $\phi_i$ to $\pi_1(S)$ is equal to that of $\phi_j$. This implies that $\phi_i$ and $\phi_j$ differ by a conjugation of an element in $\pi_1(S)$. Therefore, $\phi$ is a product of twists. So $\Ker\eta \leq \mathcal{T}(M)$.  
\end{proof}
Geometrically, Theorem~\ref{Out(pi_1(M))} says that an element $\Out(\pi_1(M))$ is a composition of twists and turns (on the level of fundamental groups). In the next section, we will prove that one can realize these by actual diffeomorphisms of $M$. The diffeomorphisms corresponding to twist will be very much like Dehn twists in surface topology in the sense that they are the identity on most of $M$, and the twisting happens on a cylindrical region of $M$ which is an isotopy of the cross section that moves a point around a nontrivial loop. 

\section{Smooth rigidity of piecewise rank $1$ manifolds}\label{smooth rigidity}
We will prove Theorem \ref{Mostow}, Theorem \ref{Nielsen} and then Theorem \ref{weak rigidity} in this section.
\begin{proof}[Proof of Theorem \ref{Mostow}]
Let $\phi \colon \pi_1(M) \longrightarrow \pi_1(M)$ be an isomorphism.  We are going to construct a diffeomorphism of $M$ by firstly defining diffeomorphisms on each piece in the cusp decomposition of $M$ and then gluing them together in such a way that the resulting diffeomorphism induces the isomorphism $\phi$. 

Let $M_i$, for $i$ in some index set $I$, be the pieces in the cusp decomposition of $M$. By Lemma~\ref{vertex subgroup}, the map $\phi$ defines a bijection $\alpha \colon I \longrightarrow I$ such that $\phi(\pi_1(M_i)) = \pi_1(M_{\alpha(i)})$ up to conjugation. By the Mostow-Prasad Rigidity Theorem for finite-volume, complete, locally symmetric, negatively curved manifolds (\cite{Prasad}), the restriction of $\phi$ on each vertex group of $\pi_1(M)$ up to conjugation is induced by an isometry $f_i$ from $M_i$ to $M_{\alpha(i)}$.

At this point one may want to glue the isometries $f_i$ together and claim that it is the desired diffeomorphism of $M$. However, there are two problems. One is that the gluing of $f_i$ at each pair of boundaries of $M_i$'s that are glued together might not be compatible with the gluing of $M_j$. The other problem is that if $M_i$ and $M_j$ are adjacent pieces, it may happen that ${f_i}_*$ and ${f_j}_*$ define different isomorphisms (by a conjugate) on the fundamental group of the boundary shared by $M_i$ and $M_j$. We might need to do some modifications to $f_i$'s near the boundary of $M_i$ before gluing them together.

We observe that the restriction of $f_i$ and $f_j$ to each pair of boundary components of $M_i$ and $M_j$ that are identified, say $S$, in the decomposition of $M$ induces the same map from $\pi_1(S)$ to $\phi(\pi_1(S))$ up to conjugation. Since $S$ is aspherical, it follows that $f_i$ and $f_j$ are homotopic. Then by Theorem~\ref{cusp isotopy}, we can modify the maps $f_i$ and $f_j$ by an isotopy of $S$ on a tubular neighborhood of the corresponding boundary component of $M_i$ and $M_j$ that is compatible with the gluing of $N$ and the action of $\phi$ on fundamental groups. Let $f \colon M \longrightarrow N$ be the diffeomorphisms obtained by gluing the modified $f_i$. Then $f_* = \phi$.
\end{proof}

\begin{theorem}\label{cusp isotopy}
Let $M$ be a piecewise rank $1$ manifold that has a cusp subgroup $C$. Let $S$ be a horo cross section of the cusp corresponding $C$ and let $P$ be the parabolic corresponding to $C$. If $\widetilde{f}\in P$ and $\widetilde{f} \colon \widetilde{S} \longrightarrow \widetilde{S}$ is $\pi_1(S)$-equivariant that is $\pi_1(S)$-equivariantly homotopic to the identity map of $\widetilde{S}$, then $\widetilde{f}$ is $\pi_1(S)$-equivariantly isotopic to the identity map. Thus, if $g \in Z(\pi_1(S))$, then there is an isotopy on $S$ that moves the base point around a loop that is homotopic to $g$. 
\end{theorem}

\begin{proof}
We have, as before 
\[ 1 \longrightarrow F \longrightarrow C \longrightarrow B \longrightarrow 1,\]
where $F$ is the torsion free, discrete, cocompact subgroup of the isometry group of a simply connected Riemannian nilpotent Lie group $H$ (i.e. $H$ is endowed with a left-invariant Riemannian metric), which is the unipotent radical of the parabolic $P$ corresponding to $C$. 

If $f$ is an isometry of $S$ that is homotopic to the identity map, then if $\widetilde{f}$ is a lift of $f$ to the universal cover $\widetilde{S}$ of $S$, then $\widetilde{f}$ commutes with all the deck transformation of $\widetilde{S}$. By the proof of Property (e), the map $\widetilde{f}$ is in the center of the unipotent radical $N_P$ of $P$.

Since, $N_P$ is a nilpotent, simply connected Lie group, the exponential map is a global diffeomorphism from $\mathfrak{n}_P$ to $N_P$. In logarithmic coordinates of $N_P$, one can take the straight line in $N_P$ from the identity to $\widetilde{f}$. 
Since $\widetilde{f}_t$ is a straight line path in the Lie algebra of the nilpotent Lie group $N_P$, and $N_P$ is normal in $P$, it follows that conjugation by $\gamma \in C$ preserves this path. Since $\widetilde{f}_1 = \widetilde{f}$ commutes with $C$, it follows that the same is try for $\widetilde{f}_t$ for a given $t \in [0,1]$. That is, 
\[ \widetilde{f}_t \circ \gamma = \gamma \circ \widetilde{f}_t, \qquad \text{for all} \; \gamma \in C.\]
So multiplication by $\widetilde{f}_t$ in $M_PN_P$ is an $C$-equivariant isotopy. Hence, it descends to an isotopy on $S$.
\end{proof}

Now we prove Theorem \ref{Nielsen}.

\begin{proof}[Proof of Theorem \ref{Nielsen}]
We construct a homomorphism $\sigma \colon \Out(\pi_1(M)) \longrightarrow \Diff(M)$ such that $\pi\circ\sigma = \rho$. In this proof, instead of consider $M$, we consider a manifold $N$ that is diffeomorphic to $M$ that is obtained as follows. Take all the pieces $M_i$'s in the cusp decomposition of $M$. We can assume that the pieces $M_i$ are such that such that the number of isometry classes of the horo boundary components of $M_i$ is minimal. This is equivalent to saying that if $\varphi \in \Out(\pi_1(M))$, then for each piece $M_i$, the restriction of $\varphi$ to $\pi_1(M_i)$ to $\varphi(\pi_1(M_i))$ corresponds to an isometry of $M_i$ to $M_{\varphi(i)}$ (where $\varphi(i)$ denotes the index of the piece whose fundamental group is the image under $\varphi$ of a conjugate of $\pi_1(M_i)$). Now to each boundary component $\partial_jM_i$, we glue $S_{ij} = \partial_jM_i\times[0,1]$ along $\partial_jM_i \times 0$ by the identity map. Then $N$ is obtained by gluing corresponding $\partial_jM_i \times 1$'s together via the gluing maps given by the decomposition of $M$ into locally symmetric pieces. Clearly, $M$ and $N$ are diffeomorphic; and $N$ is obtained by gluing the pieces of $M$ together with tubes that are diffeomorphic to $\partial_jM_i \times [-1,1]$'s.

Now we define the homomorphism $\sigma$. Let $\varphi \in \Out(\pi_1(M))$. We define $\sigma(\varphi)$ to be the diffeomorphism $f$ of $N$ such that $f$ maps each $M_i$ isometrically to $M_{\varphi(i)}$. Any twists or turns will happen in the tube connecting the boundary components of the pieces as follows. The manifold $N$ is obtained by gluing the pieces of $M$ together with the tubes that are diffeomorphic to $\partial_jM_i \times [-1,1]$'s, each cross section of which has an affine structure inherited from that of the two boundary components $\partial_jM_i \times -1$ and $\partial_jM_i \times 1$ of the tube. This p-affine structure is induced by the p-affine structure of the boundary components of the pieces they get glued to. By the proof of Theorem \ref{cusp isotopy}, the restriction of $f$ to $\partial_jM_i \times (-1)$ and that to $\partial_jM_i \times 1$ differ by an element in the center of $H$, the universal cover of $\partial_jM_i$. We define $f$ to be the straight line homotopy (as in the proof of Theorem \ref{cusp isotopy}) on $\partial_jM_i \times [-1,1]$. 

We claim that $\sigma$ is a homomorphism. This is because for each $\varphi \in \Out(\pi_1(M))$, the map $\sigma(\varphi)$ is the unique diffeomorphism that induces $\varphi$ and that acts isometrically on each $M_i$ and the restriction of which is the straight line homotopy between the map on the two ends of each tube. The composition of such two maps is a map of the same type. By uniqueness, we see that $\pi\circ\sigma = \rho$. Hence, we have proved $\Out(\pi_1(M))$ lifts to $\Diff(M)$.

If $\Out(\pi_1(M))$ is finite, by the above, it is realized by a group of diffeomorphisms $F$ of $M$. Pick any Riemannian metric $g_0$ on $M$ and average this metric by $F$ to get a metric $g$. Then $F$ is a group is a group of isometries of $M$. 
\end{proof}

Now we prove Theorem \ref{weak rigidity}.

\begin{proof}[Proof of Theorem \ref{weak rigidity}]
Suppose that $\Out(\pi_1(M))$ is infinite. Let $g$ be a Riemannian metric on $M$. Since $M$ is closed, $\Isom(M)$ is compact. We show that $\Out(\pi_1(M))$ cannot be realized as a group of isometries by showing that $\Isom(M)$ has to be discrete and thus finite.

To show that $\Isom(M)$ is discrete, we show that its identity component $\Isom^0(M)$ contains only the identity map. Since $\Isom^0(M)$ is a connected, compact Lie group, it contains an $S^1$ if it consists of more than one point, in which case there is a finite order isometry that is homotopic to the identity. However, this cannot happen because of the follow theorem.

\begin{theorem}[Borel, Conner, Raymond \cite{CR}]\label{Borel}
Let $M$ be a closed connected aspherical manifold with centerless fundamental group. If $G$ is a finite group that act effectively on $M$, then the canonical homomorphism $\psi \colon G \longrightarrow \Out(\pi_1(M))$ is a monomorphism.
\end{theorem}
We have established that $M$ is aspherical. So we are left to show that $\pi_1(M)$ is centerless.
Let $g$ be in the center of $\pi_1(M)$. Consider the action of $\pi_1(M)$ on the Bass-Serre tree $T$ of $\pi_1(M)$. Since $g$ commutes with every element of $\pi_1(M)$, it preserves their Minset in $T$. For each edge $e$ of $T$, there are elements of $\pi_1(M)$ whose Fixset is $e$. Since the action is without inversion, $g$ fixes $e$. It follows that $g$ fixes $T$ pointwise. Hence $g = 1$ and $\pi_1(M)$ is centerless.
\end{proof}

\bibliographystyle{amsplain}
\bibliography{bibliography}
\end{document}